\documentclass[a4paper,11pt,reqno]{amsart}
\usepackage{amsfonts,amsmath,amssymb,amsthm}
\usepackage{graphicx}
\usepackage{fixmath}
\usepackage{hyperref}
\usepackage[english]{babel}

\usepackage{setspace}
\usepackage{subfig}
\usepackage{url}
\usepackage{booktabs}
\usepackage{ifthen}
\usepackage{wrapfig}

\usepackage{enumerate}

\usepackage{tikz}
\usetikzlibrary{calc}
\usetikzlibrary{decorations.pathreplacing}

\usepackage[T2A,T1]{fontenc}
\newcommand{\Ya}{\mbox{\usefont{T2A}{\rmdefault}{m}{n}\CYRYA}}

\newcommand{\CC}{\mathcal{C}}

\newtheorem{thm}{Theorem}[section]
\newtheorem{lem}[thm]{Lemma}
\newtheorem{prop}[thm]{Proposition}
\newtheorem{cor}[thm]{Corollary}
\newtheorem{obs}[thm]{Observation}

\newtheorem{rem}[thm]{Remark}

\newtheorem{q}{Question}

\numberwithin{equation}{section}

\begin{document}
 
\thispagestyle{plain}

\title{On small Mixed Pattern Ramsey numbers}

\author{M. Bartlett}
\address{Marcus Bartlett  \texttt{mbartlett2@student.clayton.edu}}

\author{E. Krop}
\address{Elliot Krop  \texttt{ElliotKrop@clayton.edu}}

\author{T. Nguyen}
\address{Thuhong Nguyen \texttt{tnguyen37@student.clayton.edu}}

\author{M. Ngo}
\address{Michael Ngo  \texttt{mngo4@student.clayton.edu}}

\author{P. President}
\address{Petra President  \texttt{pmaddox2@student.clayton.edu}}
\address{Department of Mathematics \\
Clayton State University \\
Morrow, GA 30260, USA}

\begin{abstract}
We call the minimum order of any complete graph so that for any coloring of the edges by $k$ colors it is impossible to avoid a monochromatic or rainbow triangle, a Mixed Ramsey number. For any graph $H$ with edges colored from the above set of $k$ colors, if we consider the condition of excluding $H$ in the above definition, we produce a \emph{Mixed Pattern Ramsey number}, denoted $M_k(H)$. We determine this function in terms of $k$ for all colored $4$-cycles and all colored $4$-cliques. We also find bounds for $M_k(H)$ when $H$ is a monochromatic odd cycles, or a star for sufficiently large $k$. We state several open questions.
\\[\baselineskip] 
\\[\baselineskip]
\end{abstract}

\date{\today}

\maketitle

\section{Introduction} 
 
 \subsection{Colorful Patterns}

We study edge-colorings of complete graphs that avoid certain color patterns. In general, we consider colorings that avoid fixed monochromatic cliques, fixed rainbow cliques, and a multicolor pattern or family of multicolor patterns. For a given family of edge-colored graphs $\mathcal{F}$ and integers $n,p,q$, and $k$, define a $(K_p,\mathcal{F},K_q;k)$-coloring of $K_n$ to be an edge-coloring with $k$ colors, avoiding monochromatic $K_p$, rainbow $K_q$, and every member of $\mathcal{F}$. If the number of colors used is not specified and the family of colored graphs $\mathcal{F}$ is restricted to lexical graphs--the set of complete graphs of fixed order and ranked vertices, where two edges have the same color if and only if they have the same higher endpoint--then such colorings are the subject of the Erd\H os-Rado Canonical Ramsey Theorem \cite{ER}. The best current bounds on $er(p)$, the least order of complete graphs that must exhibit a monochromatic, rainbow, or lexical $K_p$, are due to Leffmann and R\"odl \cite{LR}, who showed that there exist constants $c,c'$ such that for any positive integer $p$,
\[2^{cp^2}\leq er(p)\leq 2^{c'p^2\log p}.\] 

An edge-coloring of $K_n$ is called \emph{Gallai} if no triangle is colored with three distinct colors. Let $f(s,t;k)$ be the largest $n$ so that there exists a $k$-coloring of the edges of $K_n$ where every $K_s\subseteq K_n$ has exactly $t$ different colors. For Gallai colorings avoiding monochromatic triangles, which we call the \emph{pure mixed case}, Chung and Graham \cite{CG} showed
\begin{align}
f(3,2;k)= \left\{     \begin{array}{lr}       5^{k/2} & \text{if }k\text{ is even}\\
       2\times 5^{(k-1)/2} & \text{if }k\text{ is odd}     \end{array}   \right.\label{mixed}
\end{align}
Such investigations have been generalized in various directions. For example, \emph{mixed Ramsey numbers} were studied by Axenovich and Iverson \cite{AI} and V.~Jungi\'c, T.~Kaiser, D.~Kral \cite{JKK}. In this case the problem is to determine the maximum or minimum number of colors in an edge-coloring of $K_n$, for a fixed $n$ and number of colors $k$, such that no monochromatic subgraph appears isomorphic to some graph $G$ and no rainbow subgraph appears isomorphic to some graph $H$.

Various authors, \cite{FGJM}, \cite{FM}, studied Gallai colorings excluding monochromatic cycles and paths, with the most recent paper by Hall, Magnant, Ozeki, and Tsugaki \cite{HMOT}.

In the above Ramsey problems, other than in the mixed case of monochromatic and rainbow triangles, gaps persist between upper and lower bounds. Restricting the situation to monochromatic and rainbow triangles, as well as to some predetermined pattern, we call the resulting Ramsey numbers \emph{mixed pattern Ramsey numbers}. We attempt to take the recent results of \cite{GS} and \cite{HMOT} on Gallai colorings, and apply them to the more understood small mixed cases, which were studied more generally in \cite{AJ} and \cite{AI}, to which we add particular pattern exclusions. In this limited context we are able to find many sharp bounds. However, new questions arise as to the order of the Ramsey function when forbidding various color patterns which were not previously considered, either because they could always be avoided, as in the pure Gallai case, or because questions of lexical colorings dominated.

For more on rainbow generalizations of Ramsey theory, see \cite{FMO} and the updated version \cite{FMO2}.

In this paper, we consider the problem of $(K_3,F,K_3;k)$-colorings, where $F$ is a colored cycle. In section two, we determine this function for monochromatic and bichromatic colorings of four-cycles. This problem was solved for $k=2$ by Chartrand, Kolasinski, Fujie-Okamoto, and Zhang \cite{CKFZ}. In section three, we consider the case of colorings of four-cycles by at least three colors. We rely on local arguments and elementary techniques in the previous two sections and apply the Gallai structure theorem in those that follow. In section four, we find exact bounds on mixed pattern Ramsey numbers of monochromatic odd cycles, as well as the value of the function excluding monochromatic stars, when the number of colors used is large enough. In section five we find the mixed pattern Ramsey numbers for all colorings of $K_4$. In section six, we state several open questions.

 \subsection{Definitions and Notation}
 For basic graph theoretic notation and definition see Diestel \cite{Diest}. All graphs $G$ are undirected with the vertex set $V$ and edge set $E$. 
$K_n$ denotes the complete graph on $n$ vertices. For any edge $(u,v)$, let $\CC(u,v)$ be the color on that edge, for any vertex $v$, let $\CC(v)$ be the set of colors on the edges incident to $v$, and for any edge-colored graph $H$, let $\CC(H)$ be the set of colors on the edges of $H$. We write $c(v)=\left|\CC(v)\right|$.

For any color $i$ and vertex $v$, we let $N_i(v)$ denote the set of vertices adjacent to $v$ by edges colored $i$.

For any subset of vertices $S\subseteq V$, in a colored graph $G$, we denote $[S]$ to mean the induced subgraph on $S$ with the coloring from $G$.

For any two graphs $A$ and $B$, the join of $A$ and $B$, written $A\vee B$, is the graph formed by $A$, $B$, and edges between every pair of vertices $u,v$ so that $u\in A$ and $v\in B$.

We define $C_n$ to be the monochromatic n-cycle for any integer $n>2$. In keeping with notation from \cite{CKFZ}, we define $\pi_1,\pi_2,$ and $\pi_3$ as the two-colored four-cycles in the figures below.

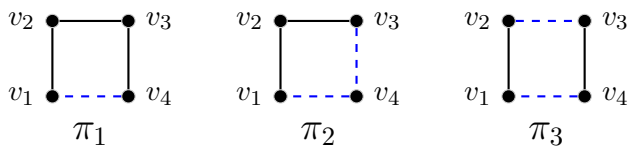
\begin{center}
\begin{tikzpicture}[>=stealth]
\tikzstyle{vertex}=[circle,minimum size=5pt,inner sep=0pt]
\tikzstyle{edge} = [draw,thick,-,black]
\tikzstyle{selected edge} = [draw,thick, -,red!50]
\tikzstyle{dash} = [draw,thick,-,style=dashed,blue]
\node[vertex] (v1) at (0,0) [label= left:$v_1$][circle,draw=black!20,fill=black!]{};
\node[vertex] (v2) at (0,1)[label=left:$v_2$][circle,draw=black!20,fill=black!]{};
\node[vertex] (v3) at (1,1)[label=right:$v_3$][circle,draw=black!20,fill=black!]{};
\node[vertex] (v4) at (1,0)[label=right:$v_4$][circle,draw=black!20,fill=black!]{};	
\draw[edge] (v1) -- (v2) -- (v3)  -- (v4);
\draw[dash] (v4) -- (v1);

\node at (.5,-.5) {\Large{$\pi_1$}};

\node[vertex] (v5) at (3,0) [label= left:$v_1$][circle,draw=black!20,fill=black!]{};
\node[vertex] (v6) at (3,1)[label=left:$v_2$][circle,draw=black!20,fill=black!]{};
\node[vertex] (v7) at (4,1)[label=right:$v_3$][circle,draw=black!20,fill=black!]{};
\node[vertex] (v8) at (4,0)[label=right:$v_4$][circle,draw=black!20,fill=black!]{};	
\draw[edge] (v5) -- (v6) -- (v7) ;
\draw[dash] (v7) -- (v8) -- (v5);

\node at (3.5,-.5) {\Large{$\pi_2$}};

\node[vertex] (u1) at (6,0) [label= left:$v_1$][circle,draw=black!20,fill=black!]{};
\node[vertex] (u2) at (6,1)[label=left:$v_2$][circle,draw=black!20,fill=black!]{};
\node[vertex] (u3) at (7,1)[label=right:$v_3$][circle,draw=black!20,fill=black!]{};
\node[vertex] (u4) at (7,0)[label=right:$v_4$][circle,draw=black!20,fill=black!]{};	
\draw[edge] (u1) -- (u2)  (u3) -- (u4) ;
\draw[dash] (u2) -- (u3)  (u4) -- (u1);

\node at (6.5,-.5) {\Large{$\pi_3$}};

\end{tikzpicture}
\end{center}

We define the rest of the multicolored four-cycles as $\pi_4, \pi_5,$ and $\pi_6$ below.

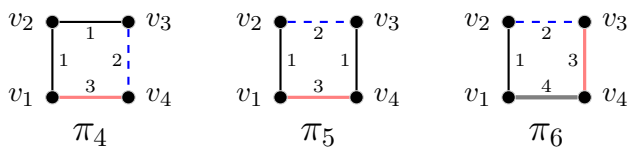
\begin{center}
\begin{tikzpicture}[>=stealth]
\tikzstyle{vertex}=[circle,minimum size=5pt,inner sep=0pt]
\tikzstyle{edge} = [draw,thick,-,black]
\tikzstyle{selected edge} = [draw,very thick, -,red!50]
\tikzstyle{dash} = [draw,thick,-,style=dashed,blue]
\tikzstyle{gray edge} = [draw,ultra thick,-,gray]

\node[vertex] (v1) at (0,0) [label= left:$v_1$][circle,draw=black!20,fill=black!]{};
\node[vertex] (v2) at (0,1)[label=left:$v_2$][circle,draw=black!20,fill=black!]{};
\node[vertex] (v3) at (1,1)[label=right:$v_3$][circle,draw=black!20,fill=black!]{};
\node[vertex] (v4) at (1,0)[label=right:$v_4$][circle,draw=black!20,fill=black!]{};	
\draw[edge] (v1) -- (v2) -- (v3);
\draw[dash] (v3) -- (v4);
\draw[selected edge] (v4) -- (v1);

\node at (.15,.5) {\tiny{$1$}};
\node at (.5,.85) {\tiny{$1$}};
\node at (.85,.5) {\tiny{$2$}};
\node at (.5,.15) {\tiny{$3$}};

\node at (.5,-.5) {\Large{$\pi_4$}};

\node[vertex] (v5) at (3,0) [label= left:$v_1$][circle,draw=black!20,fill=black!]{};
\node[vertex] (v6) at (3,1)[label=left:$v_2$][circle,draw=black!20,fill=black!]{};
\node[vertex] (v7) at (4,1)[label=right:$v_3$][circle,draw=black!20,fill=black!]{};
\node[vertex] (v8) at (4,0)[label=right:$v_4$][circle,draw=black!20,fill=black!]{};	
\draw[edge] (v5) -- (v6) (v7) -- (v8) ;
\draw[dash] (v6) -- (v7);
\draw[selected edge] (v8) -- (v5);

\node at (3.15,.5) {\tiny{$1$}};
\node at (3.5,.85) {\tiny{$2$}};
\node at (3.85,.5) {\tiny{$1$}};
\node at (3.5,.15) {\tiny{$3$}};

\node at (3.5,-.5) {\Large{$\pi_5$}};

\node[vertex] (u1) at (6,0) [label= left:$v_1$][circle,draw=black!20,fill=black!]{};
\node[vertex] (u2) at (6,1)[label=left:$v_2$][circle,draw=black!20,fill=black!]{};
\node[vertex] (u3) at (7,1)[label=right:$v_3$][circle,draw=black!20,fill=black!]{};
\node[vertex] (u4) at (7,0)[label=right:$v_4$][circle,draw=black!20,fill=black!]{};	
\draw[edge] (u1) -- (u2);
\draw[selected edge]  (u3) -- (u4);
\draw[dash] (u2) -- (u3);
\draw[gray edge] (u4) -- (u1);

\node at (6.15,.5) {\tiny{$1$}};
\node at (6.5,.85) {\tiny{$2$}};
\node at (6.85,.5) {\tiny{$3$}};
\node at (6.5,.15) {\tiny{$4$}};

\node at (6.5,-.5) {\Large{$\pi_6$}};

\end{tikzpicture}
\end{center}

 For any $k\geq j$ and $j$-colored graph $H$ of order at most $n$, we say that an edge-coloring of $K_n$ is $(K_3,H,\varphi_3;k)$ if there are $k$ colors on the edges of $K_n$ and $K_n$ does \emph{not} contain a monochromatic $K_3$, $H$, or rainbow $K_3$, which we denote by $\varphi_3$. For any positive integer $n$, an edge-coloring of $K_n$ is $(K_3,H,\varphi_3)$ if it is as above without the specification of the number of colors used. For any positive integers $n,k$ we say that $K_n$ is $(K_3,H,\varphi_3;k)$-colorable if there exists an edge coloring of $K_n$ that is $(K_3,H,\varphi_3;k)$. 

We say an edge-coloring of $K_n$ is $(K_3,\varphi_3)$, or \emph{mixed}, if $K_n$ does not contain a monochromatic or rainbow triangle. For a specified number of colors, $k$, we extend the previous definition as before. Furthermore, $R(K_3,\varphi_3;k)$ is the minimum $n$ so that every $k$-coloring of $K_n$ produces either a monochromatic or rainbow triangle as a subgraph and there exists a $(K_3,\varphi_3;k)$ coloring of $K_{n-1}$. We define $M_k(H)=R(K_3,H,\varphi_3;k)$ to be the minimum order of a complete graph such that every $k$-coloring produces either $K_3$, $H$, or $\varphi_3$ as a subgraph and there exists a $(K_3, H,\varphi_3;k)$ coloring of $K_{n-1}$. The \emph{Gallai-Ramsey number}, $GR_k(H)$, is the minimum order of a complete graph so that every edge coloring with $k$ colors and no rainbow triangles produces $H$ as a subgraph.

Let $\Ya(K_3,\varphi_3;k)$ be the maximum order of a complete graph such that every $k$-coloring produces either $K_3$ or $\varphi_3$ as a subgraph and there exists a $(K_3,\varphi_3;k)$ coloring of $K_{n+1}$. This function has been studied in a slightly different formulation in \cite{AI} and \cite{JKK}, and many asymptotic bounds are known.

For a fixed number of colors $k$, it may help to keep in mind the heuristic illustration of the relationship between values of $n$ and mixed colorings.

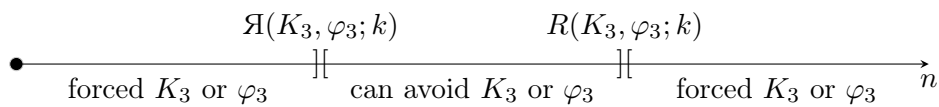
\begin{center}
\begin{tikzpicture}[>=stealth]
\tikzstyle{vertex}=[circle,minimum size=5pt,inner sep=0pt]
\tikzstyle{edge} = [draw,thick,-,black]
\tikzstyle{selected edge} = [draw,very thick, -,red!50]
\tikzstyle{dash} = [draw,thick,-,style=dashed,blue]
\tikzstyle{gray edge} = [draw,ultra thick,-,gray]

\node[vertex] (v1) at (0,0) [][circle,draw=black!20,fill=black!]{};
\node (v4) at (12,0)[][circle]{};

\draw [->] (0,0) -- (12,0) node [below] {$n$};

\node at (4,0){$][$};
\node at (8,0){$][$};


\node at (2,-.35){forced $K_3$ or $\varphi_3$};
\node at (6,-.35){can avoid $K_3$ or $\varphi_3$};
\node at (10,-.35){forced $K_3$ or $\varphi_3$};

\node at (4,.5){$\Ya(K_3,\varphi_3;k)$};
\node at (8,.5){$R(K_3,\varphi_3;k)$};

\end{tikzpicture}
\end{center}

For certain colored patterns $H$ and number of colors $k$, $K_n$ does not admit a $(K_3,H,\varphi_3;k)$ edge-coloring. In this case, it is convenient to define the minimum order complete graph with at least $k$ edges, $MIN(k)$. In particular, define $|MIN(k)|=min(k)=\left\lceil \sqrt{2k+\frac{1}{4}}+\frac{1}{2}\right\rceil$.

\subsection{Some Useful Known Results}

With the above notation, we can restate (\ref{mixed}) as
\begin{thm}\label{mixed2}
\begin{align}
R(K_3,\varphi_3;k)= \left\{     \begin{array}{ll}       5^{k/2}+1, & \text{if }k\text{ is even}\\
       2\times 5^{(k-1)/2}+1, & \text{if }k\text{ is odd}     \end{array}   \right.
\end{align}
\end{thm}

\begin{lem}\cite{AJ} \label{AJ}
Let $c$ be a coloring of $E(K_n)$ with no rainbow triangle. Then there is a vertex with at least $\frac{(n+1)}{3}$ edges incident to it of the same color.
\end{lem}


\begin{thm}\cite{FGJM}\label{FGJM}
$GR_k(C_4)=k+4$
\end{thm}

We now state the fundamental structure theorem for Gallai colorings, originally found by T. Gallai in \cite{G} and further explored in \cite{GS} and by various authors.

\begin{thm}\cite{G}\label{G}
Any Gallai-colored complete graph can be partitioned into more than one set of vertices, so that there is only one color on the edges between any pair of parts, and at most two colors on edges between parts.
\end{thm}

\begin{thm}\cite{GSSS}\label{GSSS}
Let $H$ be a fixed monochromatic graph without isolated vertices.Then $GR_k(H)$ is exponential in $k$ if $H$ is not bipartite and linear in $k$ if $H$ is bipartite and not a star.
\end{thm}

\section{Forcing Bichromatic Cycles}

\begin{thm}\label{lbpi1}
$M_k(\pi_1)>2^k$
\end{thm}

\begin{proof}
For $k=2$, we produce a $(K_3,\pi_1,\varphi_3;2)$ coloring of $K_4$. Construct a monochromatic cycle $C_4$ and color the remaining edges by the other color. Call this colored graph $G_4$. To produce $G_8$, a $(K_3,\pi_1,\varphi_3;3)$-colored $K_8$, we connect two copies of $G_4$ by edges colored by a third color. Proceeding this way, we produce a $(K_3,\pi_1,\varphi_3;k)$-colored $K_{2^k}$ by connecting two copies of $G_{2^{k-1}}$ by edges colored by a $k^{th}$ color. If $G_{2^{k-1}}$ is $(K_3,\pi_1,\varphi_3;k-1)$-colored, then we only need check that a monochromatic $K_3, \pi_1$, or $\varphi_3$ does not occur between the two copies of $G_{2^{k-1}}$, which is an easy verification.
\end{proof}

\begin{thm}\label{colorfulvertexpi1}
For any positive integers $l,k,n$ so that $l\geq k$, if $2^{k}+1\leq n \leq 2^{k+1}$, then any $(K_3,\pi_1,\varphi_3;l)$ coloring of $K_n$ produces a vertex $v$ so that $c(v)>k$.
\end{thm}

\begin{proof}
We proceed by induction on $n$. By \cite{CKFZ}, $M_2(\pi_1)=5$. We check the statement for $n=3$ and $4$. In either case, if $c(v)=1$ for every vertex $v$, then $K_n$ is monochromatic, which is not allowed in a $(K_3,\pi_1,\varphi_3;l)$ coloring. 

Next, for some $n>4$, suppose the statement true for at most $n-1$ vertices and consider $G$, a $(K_3,\pi_1,\varphi_3;l)$ colored $K_{n}$. 

If $n>2^k+1$, then choose a vertex $v$ and remove it from $G$. Notice that the remaining graph is $(K_3,\pi_1,\varphi_3)$ free, and since $n-1\geq 2^k+1$, there are at least $k+1$ colors. Thus, $G-\{v\}$ contains a vertex $u$ so that $c(u)>k$.

If $n=2^k+1$, we consider any vertex $v$ incident to edges of the same color, say $1$. If $\left|N_1(v)\right|\geq 2^{k-1}+1$, then by induction, we can find a vertex $u\in N_1(v)$ so that $c(u)\geq k$ in the induced subgraph $H=[N_1(v)]$. Restricting out attention to $H$, if $1\in \CC(u)$ in $H$, then that edge together with incident edges from $v$ would form a monochromatic triangle, which is impossible by the definition of the coloring. Thus, $c(u)\geq k+1$ in $G$. This observation leaves us to consider the case that for any vertex $v\in G$ and color $i$, $\left|N_i(v)\right|\leq 2^{k-1}$, and thus, $v$ is adjacent to at least $2^{k-1}$ vertices by edges not colored $1$. Let $I=\CC(v)\cap\CC(H)$ and call $H'$ the induced subgraph on the neighbors of $v$ incident to edges of these colors, that is $[N_I(v)]=H'$. Let $J=(\CC(v)-(\{1\}\cup I))$ and call $H''$ the induced subgraph on the neighbors of $v$ incident to edges of these colors, that is $[N_J(v)]=H''$.

Lemma \ref{AJ} guarantees the existence of $v$ so that $\left|N_i(v)\right|\geq\frac{n+1}{3}$ for some color $i$. Without loss of generality, say $1=i$ and notice that by induction, $H=[N_1(v)]$ must contain a vertex $u$ so that $c(u)>\left\lfloor\log_2(\frac{n+1}{3})\right\rfloor=\left\lfloor\log_2(\frac{2}{3})+\log_2(2^{k-1}+1)\right\rfloor\geq \left\lfloor\ k-1.585\right\rfloor=k-2$.

Next, consider any vertex $w\in H'$. Notice $c(vw)=i$ and $c(uw)=i$ produces a $\pi_1$ subgraph. Assigning $c(uw)=j$ for some color $j\notin \{1,i\}$, produces $\varphi_3$. Thus, we are forced to conclude that $c(uw)=1$.

For any $x\in H''$ so that $c(vx)=j$ for some color $j$, if $c(ux)=j$, then together with the color $1$ on the edge $vu$, $c(u)\geq k+1$. To avoid a rainbow triangle, $c(ux)\in \{1,j\}$. Thus, we are left with the case where $c(ux)=1$.

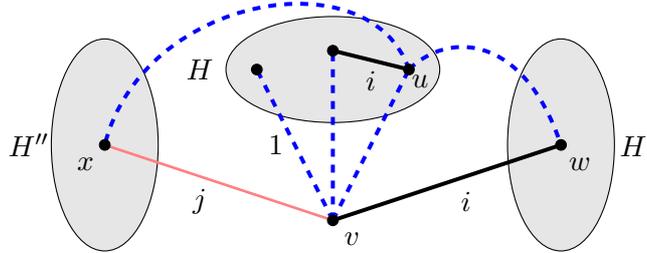
\begin{center}

\begin{tikzpicture}[>=stealth]

\filldraw[fill=gray!20!white] (0,1) ellipse (40pt and 20pt);
\filldraw[fill=gray!20!white] (3,0) ellipse (20pt and 40pt);
\filldraw[fill=gray!20!white] (-3,0) ellipse (20pt and 40pt);

\draw [line width=1.5, style=dashed, color=blue] (0,-1) -- (1,1);
\draw [line width=1.5, style=dashed, color=blue] (0,-1) -- (-1,1);
\draw [line width=1.5, style=dashed, color=blue] (0,-1) -- (0,1.25);
\draw [line width=1.5, color=black] (1,1) -- (0,1.25);
\draw [line width=1.5, color=black] (0,-1) -- (3,0);
\draw [line width=1.5, style=dashed, color=blue] (1,1) .. controls (1.5,1.5) and (2.5,1.5) .. (3,0);
\draw [line width=1, color=red!50] (0,-1) -- (-3,0);
\draw [line width=1.5, style=dashed, color=blue] (1,1) .. controls (.5,2.5) and (-2.5,2) .. (-3,0);

\filldraw
(0,-1) circle (2pt)
(1,1) circle (2pt)
(-1,1) circle (2pt)
(0,1.25) circle (2pt)
(3,0) circle (2pt)
(-3,0) circle (2pt)
;
 
\node at (.25,-1.25)
{$v$};

\node at (-1.75,1)
{$H$};

\node at (-.75,0)
{$1$};

\node at (.5,.85)
{$i$};

\node at (1.75,-.75)
{$i$};

\node at (1.15,.85)
{$u$};

\node at (3.25,-.25)
{$w$};

\node at (4,0)
{$H'$};

\node at (-4,0)
{$H''$};

\node at (-3.25,-.25)
{$x$};

\node at (-1.75,-.75)
{$j$};




\end{tikzpicture}

\end{center}

Since $c(uv)=1$ and $c(uw)=1$ for any $w \notin H$, $\left|N_1(u)\right|\geq 2^{k-1}+1$ and by induction, we can find a vertex $x\in [N_1(u)]$ which is incident to at least $k$ edges of different colors in $[N_1(u)]$, none of which are $1$. Since $c(ux)=1$, $c(x)\geq k+1$ in $G$.
\end{proof}

\begin{cor}\label{pi1}
$M_k(\pi_1)=2^k+1$
\end{cor}

\begin{thm}\label{pi2}
\[ M_k(\pi_2) = 
    \left\{\begin{array}{ll}
    6, & \mbox{ when } k=2\\
    5, & \mbox{ when } k=3\\
    min(k), & \mbox{ when } k\geq4
    \end{array}\right. \]
\end{thm}

\begin{proof}
We produce a $(K_3,\pi_2,\varphi_3;2)$-colored $K_5$ and a $(K_3,\pi_2,\varphi_3;3)$-colored $K_4$ below.

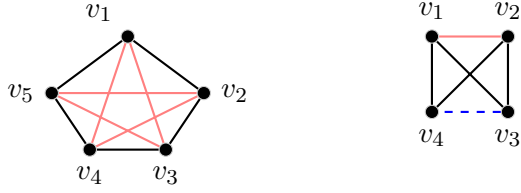
\begin{center}

\begin{tikzpicture}[>=stealth]
	\tikzstyle{vertex}=[circle,minimum size=5pt,inner sep=0pt]
	\tikzstyle{edge} = [draw,thick,-,black]
	\tikzstyle{selected edge} = [draw,thick, -,red!50]
	\tikzstyle{dash} = [draw,thick,-,style=dashed,blue]
	
	\node[vertex] (v1) at (0,1) 
[label=above left:$v_1$][circle,draw=black!20,fill=black!]{};
	\node[vertex] (v2) at (1,.25)
[label=right:$v_2$][circle,draw=black!20,fill=black!]{};
	\node[vertex] (v3) at (.5,-.5)
[label=below:$v_3$][circle,draw=black!20,fill=black!]{};
	\node[vertex] (v4) at (-.5,-.5)
[label=below:$v_4$][circle,draw=black!20,fill=black!]{};
	\node[vertex] (v5) at (-1,.25)
[label=left:$v_5$][circle,draw=black!20,fill=black!]{};
\draw[edge] (v1) -- (v2) -- (v3) -- (v4) -- (v5) -- (v1);
\draw[selected edge] (v1) -- (v3) -- (v5) -- (v2) -- (v4) -- (v1);

	\node[vertex] (v1) at (4,1) 
[label=above:$v_1$][circle,draw=black!20,fill=black!]{};
	\node[vertex] (v2) at (5,1)
[label=above:$v_2$][circle,draw=black!20,fill=black!]{};
	\node[vertex] (v3) at (5,0)
[label=below:$v_3$][circle,draw=black!20,fill=black!]{};
	\node[vertex] (v4) at (4,0)
[label=below:$v_4$][circle,draw=black!20,fill=black!]{};
	
\draw[edge] (v2) -- (v3) -- (v1) -- (v4) -- (v2);
\draw[selected edge] (v1) -- (v2);
\draw[dash] (v3) -- (v4);
\end{tikzpicture}

\end{center}

It is shown in \cite{CKFZ} that $M_2(\pi_2)=6$.

If the order of the graph is at least six, Lemma \ref{AJ} guarantees a vertex $v$ so that $\left|N_i(v)\right|\geq\frac{n+1}{3}$ for some color $i$. Therefore, if $n>5$, $\left|N_i(v)\right|\geq 3$, which produces $\pi_2$ as illustrated below.

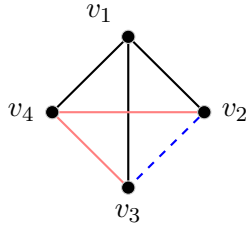
\begin{center}
\begin{tikzpicture}[>=stealth]
	\tikzstyle{vertex}=[circle,minimum size=5pt,inner sep=0pt]
	\tikzstyle{edge} = [draw,thick,-,black]
	\tikzstyle{selected edge} = [draw,thick, -,red!50]
	\tikzstyle{dash} = [draw,thick,-,style=dashed,blue]
	\node[vertex] (v1) at (0,0) [label=above left:$v_1$][circle,draw=black!20,fill=black!]{};
	\node[vertex] (v2) at (1,-1)[label=right:$v_2$][circle,draw=black!20,fill=black!]{};
	\node[vertex] (v3) at (0,-2)[label=below:$v_3$][circle,draw=black!20,fill=black!]{};
	\node[vertex] (v4) at (-1,-1)[label=left:$v_4$][circle,draw=black!20,fill=black!]{};	
\draw[edge] (v1) -- (v2)  (v1) -- (v3)  (v1) -- (v4);
\draw[selected edge] (v2) -- (v4) (v4) -- (v3);
\draw[dash] (v2) -- (v3);
\end{tikzpicture}
\end{center}


If $n=k=4$ and for some vertex $v$, $c(v)=3$, then color $4$ cannot be placed on any remaining edge without creating $\varphi_3$, and $c(v)=2$ for all $v$. Suppose the edges incident to vertices $v_1$ and $v_2$ do not share a color pair, and say $c(v_1v_2)=1, c(v_1v_4)=2, c(v_2v_3)=3$. Notice that color $4$ cannot be placed on any remaining edge without creating $\varphi_3$. Thus, we are left with pairs of vertices that share a color pair on incident edges, illustrated in the figure below.

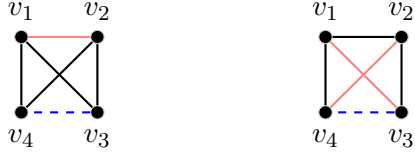
\begin{center}

\begin{tikzpicture}[>=stealth]
	\tikzstyle{vertex}=[circle,minimum size=5pt,inner sep=0pt]
	\tikzstyle{edge} = [draw,thick,-,black]
	\tikzstyle{selected edge} = [draw,thick, -,red!50]
	\tikzstyle{dash} = [draw,thick,-,style=dashed,blue]
	
	\node[vertex] (v1) at (0,1) 
[label=above:$v_1$][circle,draw=black!20,fill=black!]{};
	\node[vertex] (v2) at (1,1)
[label=above:$v_2$][circle,draw=black!20,fill=black!]{};
	\node[vertex] (v3) at (1,0)
[label=below:$v_3$][circle,draw=black!20,fill=black!]{};
	\node[vertex] (v4) at (0,0)
[label=below:$v_4$][circle,draw=black!20,fill=black!]{};

\draw[edge] (v2) -- (v3) -- (v1) -- (v4) -- (v2);
\draw[selected edge] (v1) -- (v2);
\draw[dash] (v3) -- (v4);	

	\node[vertex] (v1) at (4,1) 
[label=above:$v_1$][circle,draw=black!20,fill=black!]{};
	\node[vertex] (v2) at (5,1)
[label=above:$v_2$][circle,draw=black!20,fill=black!]{};
	\node[vertex] (v3) at (5,0)
[label=below:$v_3$][circle,draw=black!20,fill=black!]{};
	\node[vertex] (v4) at (4,0)
[label=below:$v_4$][circle,draw=black!20,fill=black!]{};
	
\draw[edge] (v4) -- (v1) -- (v2) -- (v3);
\draw[selected edge] (v1) -- (v3) (v2) -- (v4);
\draw[dash] (v3) -- (v4);

\end{tikzpicture}

\end{center}

Notice that the color of the edge $v_3v_4$ in the second figure cannot be a color different from the colors present on the edges between $v_1,v_2,v_3,v_4$, without producing $\varphi_3$. Thus, in either figure, at most three colors are used, which contradicts that $k=4$.

Observe that if $k>4$ and $n\leq 4$, then $\varphi_3$ cannot be avoided.

For the remaining cases of any $(K_3,\pi_2,\varphi_3)$-colored $K_5$ by at least three colors, notice by the figure above, that for any vertex $v$, $c(v)\in \{2,3,4\}$. We label the remaining vertices $v_1,v_2,v_3,v_4$. If $c(v)>2$, say $\CC(vv_1)=1,\CC(vv_2)=2,\CC(vv_3)=3$ and without loss of generality let $\CC(v_1v_2)=1$. Notice that $\CC(v_2v_3)\in \{2,3\}$ and the argument is identical for either color. Thus, let $\CC(v_2v_3)=2$ and notice that we are forced to conclude that $\CC(v_1v_3)=1$. However, this produces $\pi_2$ on the vertices $v,v_1,v_3,v_2$.

For the last case, notice that for every vertex $v$, $c(v)=2$, otherwise we refer to a previous case. If $k=3$, then there are exactly ${3\choose 2}=3$ ways to pick two color classes of edges at any vertex. However, there are five vertices, so by the pigeonhole principle at least two vertices share a pair of color classes. Call two such vertices $v_1$ and $v_2$ and let $c(v_1v_2)=1$. We illustrate the coloring in the following figure:

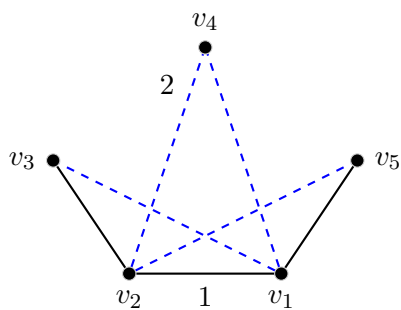
\begin{center}

\begin{tikzpicture}[>=stealth,scale=2]
	\tikzstyle{vertex}=[circle,minimum size=5pt,inner sep=0pt]
	\tikzstyle{edge} = [draw,thick,-,black]
	\tikzstyle{red} = [draw,thick, -,red!50]
  \tikzstyle{dash} = [draw,thick,-,style=dashed,blue]
	\node[vertex] (v4) at (0,1) 
[label=above:$v_4$][circle,draw=black!20,fill=black!]{};
	\node[vertex] (v5) at (1,.25)
[label=right:$v_5$][circle,draw=black!20,fill=black!]{};
	\node[vertex] (v1) at (.5,-.5)
[label=below:$v_1$][circle,draw=black!20,fill=black!]{};
	\node[vertex] (v2) at (-.5,-.5)
[label=below:$v_2$][circle,draw=black!20,fill=black!]{};
	\node[vertex] (v3) at (-1,.25)
[label=left:$v_3$][circle,draw=black!20,fill=black!]{};
\draw[edge] (v1) -- (v2)  (v2) -- (v3) (v1) -- (v5);
\draw[dash] (v1) -- (v4)  (v1) -- (v3) (v2) -- (v4)  (v2) -- (v5);

\node at (0,-.65)
{$1$};

\node at (-.25,.75)
{$2$};

\end{tikzpicture}

\end{center}

Notice that edges between the vertices $v_3,v_4,v_5$ can only be colored by $1$ or $2$ to avoid rainbow triangles. However, this means that $k=2$, which is a contradiction.

If $k>3$, and no two vertices have the same pair of color classes, then the coloring forces a rainbow triangle as shown below:

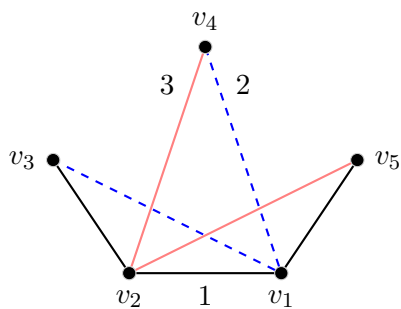
\begin{center}

\begin{tikzpicture}[>=stealth,scale=2]
	\tikzstyle{vertex}=[circle,minimum size=5pt,inner sep=0pt]
	\tikzstyle{edge} = [draw,thick,-,black]
	\tikzstyle{red} = [draw,thick, -,red!50]
  \tikzstyle{dash} = [draw,thick,-,style=dashed,blue]
	\node[vertex] (v4) at (0,1) 
[label=above:$v_4$][circle,draw=black!20,fill=black!]{};
	\node[vertex] (v5) at (1,.25)
[label=right:$v_5$][circle,draw=black!20,fill=black!]{};
	\node[vertex] (v1) at (.5,-.5)
[label=below:$v_1$][circle,draw=black!20,fill=black!]{};
	\node[vertex] (v2) at (-.5,-.5)
[label=below:$v_2$][circle,draw=black!20,fill=black!]{};
	\node[vertex] (v3) at (-1,.25)
[label=left:$v_3$][circle,draw=black!20,fill=black!]{};
\draw[edge] (v1) -- (v2)  (v2) -- (v3) (v1) -- (v5);
\draw[dash] (v1) -- (v4)  (v1) -- (v3);
\draw[red]  (v2) -- (v4)  (v2) -- (v5);

\node at (0,-.65)
{$1$};

\node at (-.25,.75)
{$3$};

\node at (.25,.75)
{$2$};
\end{tikzpicture}

\end{center}

Thus, some pair of vertices must have the same pair of color classes, which returns us to the case when $k=2$. 

\end{proof}

\begin{thm}\label{pi3}
$M_k(\pi_3)=k+2$
\end{thm}

\begin{proof}
To show the lower bound, we color $K_{k+1}$ by a lexical ``greedy coloring". Label the vertices of $K_{k+1}$ by $\{v_1, \dots, v_{k+1}\}$ and color the edges incident to $v_1$ by color $1$. Next, color the edges incident to $v_2$ which have not been previously colored, by color $2$. Continue this way, until all edges are colored. 

To show the upper bound, let $G$ be a $(K_3,\pi_3,\varphi_3)$-colored $K_{k+2}$. Choose a vertex $v_1$ and for any color $i$, let $S_i=N_i(v_1)$. We argue that for any two colors $i,j$, no color of an edge in $[S_i]$ can be repeated on edges in $[S_j]$. Consider the figure below where $\{v_2,v_3\}\subseteq S_1$ and $\{v_4,v_5\}\subseteq S_2$, and colors $1,2$, and $3$ are as labeled. 

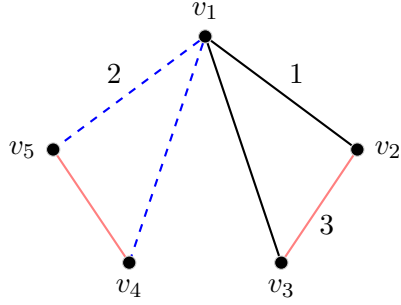
\begin{center}

\begin{tikzpicture}[>=stealth,scale=2]
	\tikzstyle{vertex}=[circle,minimum size=5pt,inner sep=0pt]
	\tikzstyle{edge} = [draw,thick,-,black]
	\tikzstyle{red} = [draw,thick, -,red!50]
  \tikzstyle{dash} = [draw,thick,-,style=dashed,blue]
	\node[vertex] (v1) at (0,1) 
[label=above:$v_1$][circle,draw=black!20,fill=black!]{};
	\node[vertex] (v2) at (1,.25)
[label=right:$v_2$][circle,draw=black!20,fill=black!]{};
	\node[vertex] (v3) at (.5,-.5)
[label=below:$v_3$][circle,draw=black!20,fill=black!]{};
	\node[vertex] (v4) at (-.5,-.5)
[label=below:$v_4$][circle,draw=black!20,fill=black!]{};
	\node[vertex] (v5) at (-1,.25)
[label=left:$v_5$][circle,draw=black!20,fill=black!]{};
\draw[edge] (v1) -- (v2)  (v1) -- (v3);
\draw[dash] (v1) -- (v4)  (v1) -- (v5);
\draw[red] (v2) -- (v3) (v4) -- (v5);

\node at (.6,.75)
{$1$};

\node at (-.6,.75)
{$2$};

\node at (.8,-.25)
{$3$};
\end{tikzpicture}

\end{center}

Without loss of generality, suppose $c(v_2v_5)=1$. To avoid a rainbow triangle on $v_1,v_2,v_4$ and on $v_2,v_4,v_5$, notice that $c(v_2v_4)=1$. To avoid rainbow triangles on $v_2,v_3,v_5$ and $v_1,v_3,v_5$, we have $c(v_3v_5)=1$. To avoid rainbow triangles in induced subgraphs on $v_1,v_3,v_4$ and $v_3,v_4,v_5$, $c(v_3v_4)=1$. However, notice that this produces $\pi_3$ in the induced subgraph on $v_2,v_5,v_4,v_3$.

Next, we show that no color of $\CC(v_1)$ can be repeated in $[S_i]$ for any $i\in \{1, \dots, c(v)\}$. Suppose in the above figure, $\CC(v_2v_3)=2$. If $c(v_2v_5)=1$ or $c(v_3v_5)=1$, then we produce $\pi_3$ in the induced subgraph with vertices $v_1,v_2,v_3,v_5$. However, the only color available for those edges is $2$, which produces a monochromatic triangle.

With these properties we have
\begin{align}
c(G)\geq c(v)+\sum_{i=1}^{c(v)}c([S_i]) \label{1}
\end{align}
For the case of two colors, \cite{CKFZ} showed that $GR(K_3,\pi_3,\varphi_3;2)=4$. By induction on the order of the graph, we assume that for some integer $N>4$, complete graphs of order $n<N$ can be colored by $n-1$ colors avoiding monochromatic triangles, rainbow triangles, and $\pi_3$, but not by $n-2$ colors. Absorbing the count of the colors at $v$ into the sum in (\ref{1}), we have
\[c(G)\geq \sum_{i=1}^{c(v)}{|S_i|}=N-1.\]
Setting $N=k+2$ shows that $GR(K_3,\pi_3,\varphi_3;k)\leq k+2$.
\end{proof}

\section{Forcing Multi-chromatic Four-Cycles}

\begin{thm}\label{pi4}
\[ M_k(\pi_4) = 
    \left\{\begin{array}{ll}
    5, & \mbox{ when } k=3\\
    min(k), & \mbox{ when } k\geq 4
    \end{array}\right. \]
\end{thm}

\begin{proof}
If $k=3$, we produce a $(K_3,\pi_4,\varphi_3)$ coloring of $K_4$ by the figure below.

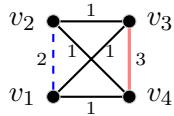
\begin{center}
\begin{tikzpicture}[>=stealth]
\tikzstyle{vertex}=[circle,minimum size=5pt,inner sep=0pt]
\tikzstyle{edge} = [draw,thick,-,black]
\tikzstyle{selected edge} = [draw,very thick, -,red!50]
\tikzstyle{dash} = [draw,thick,-,style=dashed,blue]

\node[vertex] (v1) at (0,0) [label= left:$v_1$][circle,draw=black!20,fill=black!]{};
\node[vertex] (v2) at (0,1)[label=left:$v_2$][circle,draw=black!20,fill=black!]{};
\node[vertex] (v3) at (1,1)[label=right:$v_3$][circle,draw=black!20,fill=black!]{};
\node[vertex] (v4) at (1,0)[label=right:$v_4$][circle,draw=black!20,fill=black!]{};	
\draw[edge] (v1) -- (v3) -- (v2) -- (v4) -- (v1);
\draw[dash] (v1) -- (v2);
\draw[selected edge] (v3) -- (v4);

\node at (-.15,.5) {\tiny{$2$}};
\node at (.5,1.15) {\tiny{$1$}};
\node at (1.15,.5) {\tiny{$3$}};
\node at (.5,-.15) {\tiny{$1$}};
\node at (.25,.6) {\tiny{$1$}};
\node at (.75,.6) {\tiny{$1$}};

\end{tikzpicture}
\end{center}

For $n\geq 6$, we apply Lemma \ref{AJ} and find a vertex $v$ such that $\left|N_i(v)\right|\geq \frac{n+1}{3}\geq 3$ for some color $i$. Without loss of generality, suppose $i=1$ and let $u_1,u_2,u_3 \in N_1(v)$. Notice that the triangle on these vertices is in two colors. Say $c(u_1u_2)=c(u_1u_3)=2$ and $c(u_2u_3)=3$. However, the cycle $vu_1u_2u_3$ forms $\pi_4$.

If $k>3$, we show that $M_k(\pi_4)\leq 4$ by the next observations about possible $(K_3,\pi_4,\varphi_3;k)$ colorings of $K_4$.
\begin{enumerate}
	\item No vertex can  be incident to all edges of the same color by the initial argument using Lemma \ref{AJ}.
	\item No vertex can be incident to three edges of different colors as no remaining edge could be colored by a fourth color without forming $\varphi_3$.
	\item Every vertex must be incident to three edges of two colors.
\end{enumerate}

Consider the following coloring

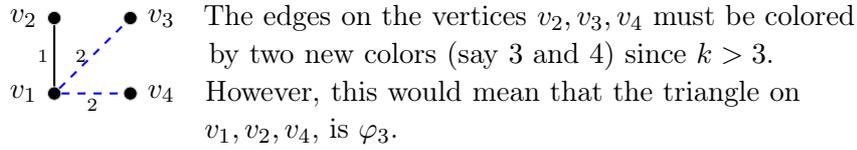
\begin{tikzpicture}[>=stealth]
\tikzstyle{vertex}=[circle,minimum size=5pt,inner sep=0pt]
\tikzstyle{edge} = [draw,thick,-,black]
\tikzstyle{selected edge} = [draw,very thick, -,red!50]
\tikzstyle{dash} = [draw,thick,-,style=dashed,blue]

\node[vertex] (v1) at (0,0) [label= left:$v_1$][circle,draw=black!20,fill=black!]{};
\node[vertex] (v2) at (0,1)[label=left:$v_2$][circle,draw=black!20,fill=black!]{};
\node[vertex] (v3) at (1,1)[label=right:$v_3$][circle,draw=black!20,fill=black!]{};
\node[vertex] (v4) at (1,0)[label=right:$v_4$][circle,draw=black!20,fill=black!]{};	
\draw[edge] (v1) -- (v2);
\draw[dash] (v1) -- (v4) (v1) -- (v3);

\node at (-.15,.5) {\tiny{$1$}};
\node at (.35,.5) {\tiny{$2$}};
\node at (.5,-.15) {\tiny{$2$}};


\node at(6.25,1) {The edges on the vertices $v_2,v_3,v_4$ must be colored};
\node at (5.75,.5) {by two new colors (say $3$ and $4$) since $k>3$.};
\node at (5.90,0) {However, this would mean that the triangle on};
\node at (3.25,-.5) {$v_1,v_2,v_4$, is $\varphi_3$.};

\end{tikzpicture}

Thus, we are left to consider colorings of $K_5$. From the above condition, $(1)$ holds. Furthermore, if $c(v)>2$, say $c(vv_1)=1,$ $c(vv_2)=2,$ and $c(vv_3)=3$, then without loss of generality, let $c(v_2v_3)=2$. Notice that $c(v_1v_2)\in\{1,2\}$ and either choice would produce $\pi_4$. Thus, $c(v)=2$ for all $v\in V$. The rest of the argument is identical to the proof for $M_k(\pi_2)$. 
\end{proof}

\begin{thm}\label{pi5}
$M_k(\pi_5)=k+4$
\end{thm}

\begin{proof}
The lower bound for $k=2$ is attained by coloring a $5$-cycle in $K_5$ by one color and the remaining edges by another color. For every additional color, introduce a new vertex $u$ and color all edges from $u$ to the rest of the graph by a new color. Since such a coloring of $K_n$ contains no independent monochromatic pair of edges, $\pi_5$ never appears and it is easy to verify the same for $K_3$ and $\varphi_3$.

The argument for the upper bound is similar to that for $\pi_3$. Let $G$ be a $(K_3,\pi_5,\varphi_3)$ colored $K_{k+4}$. Choose a vertex $v_1$ and for any color $i$, let $S_i=N_i(v_1)$. We argue that for any color $i$, the colors of edge in $[S_i]$ must be the colors of $\CC(v_1)$. Consider the figure below where $\{v_2,v_3\}\subseteq S_1$ and $\{v_4,v_5\}\subseteq S_2$, and colors $1,2,x$, and $y$ are as labeled. 

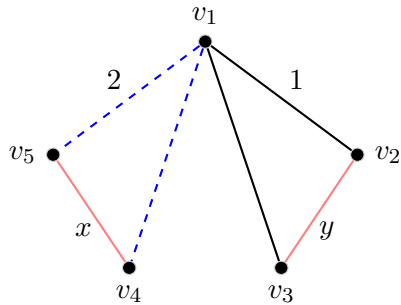
\begin{center}

\begin{tikzpicture}[>=stealth,scale=2]
	\tikzstyle{vertex}=[circle,minimum size=5pt,inner sep=0pt]
	\tikzstyle{edge} = [draw,thick,-,black]
	\tikzstyle{red} = [draw,thick, -,red!50]
  \tikzstyle{dash} = [draw,thick,-,style=dashed,blue]
	\node[vertex] (v1) at (0,1) 
[label=above:$v_1$][circle,draw=black!20,fill=black!]{};
	\node[vertex] (v2) at (1,.25)
[label=right:$v_2$][circle,draw=black!20,fill=black!]{};
	\node[vertex] (v3) at (.5,-.5)
[label=below:$v_3$][circle,draw=black!20,fill=black!]{};
	\node[vertex] (v4) at (-.5,-.5)
[label=below:$v_4$][circle,draw=black!20,fill=black!]{};
	\node[vertex] (v5) at (-1,.25)
[label=left:$v_5$][circle,draw=black!20,fill=black!]{};
\draw[edge] (v1) -- (v2)  (v1) -- (v3);
\draw[dash] (v1) -- (v4)  (v1) -- (v5);
\draw[red] (v2) -- (v3) (v4) -- (v5);

\node at (.6,.75)
{$1$};

\node at (-.6,.75)
{$2$};

\node at (-.8,-.25)
{$x$};

\node at (.8,-.25)
{$y$};
\end{tikzpicture}

\end{center}

Notice that $c(v_3v_4)\in \{1,2\}$ and either choice produces $\pi_5$ when $x,y\notin\{1,2\}$. If $x=1$ and $y\neq 2$, then $c(v_3v_4)=1$ produces $\pi_5$ on the edges of $v_1v_2v_3v_4$, so $c(v_3v_4)=2$. If $c(v_2v_5)=2$, then we have $\pi_5$ on the edges of $v_2v_3v_4v_5$, which forces $c(v_2v_5)=1$. However, this leaves us with no choice for a color on $v_3v_5$ because $c(v_3v_5)=2$ produces $\varphi_3$ on the edges of $v_2v_3v_5$ and $c(v_3v_5)=1$ produces $\pi_5$ on the edges of $v_1v_2v_3v_5$.

Thus, we conclude that the colors of edge in $[S_i]$ must be the colors of $\CC(v_1)$. Moreover, for any colors $i,j\in \CC(v_1)$ so that $i$ and $j$ occur at least twice on edges incident to $v_1$, $\CC([S_i])=\{j\}$. We call this the \emph{copycat condition}.

Since there are $k$ colors on the edges of $G$, all of which are represented on the edges incident to $v_1$, there are three edges incident to $v_1$ left to consider. 

If the colors on those three edges are all different, say $\{1,2,3\}$, then the copycat condition cannot hold. 

If the colors contain one repetition, say $\{1,1,2\}$, then $[N_1(v_1)]$ contains a triangle in two colors, one of which is not $2$, breaking the copycat condition.

If the colors on those three edges are all the same, say $\{1,1,1\}$, $[N_1(v_1)]$ contains a triangle in two colors, say $2$ and $3$ where $3$ appears once. Call $v_2$ the vertex where $c(v_1v_2)=2$, and for $i=3,4,5$, call $v_i$ be such that $c(v_1v_i)=1$. Notice $c(v_2v_3)=2$ to avoid $\pi_5$ on the edges of $v_1,v_2,v_3,v_4$, $c(v_2v_4)=2$ to avoid $\pi_5$ on the edges of $v_1v_2v_4v_5$ and to avoid $\varphi_3$ on the edges of $v_2v_3v_4$. However, this leaves no color available for $v_2v_5$.
\end{proof}

\begin{cor}
$M_k(C_4)=k+4$.
\end{cor}

\begin{proof}
We apply the lower bound construction from Theorem \ref{pi5} and the upper bound from Theorem \ref{FGJM}. 
\end{proof}

\begin{obs}\label{pi6}
For any $k\geq 2$, $M_k(\pi_6)=R(K_3,\varphi_3;k)$.
\end{obs}

\begin{proof}
Notice that in the figure of $\pi_6$ defined above, any coloring of $v_1v_3$ produces a rainbow triangle. Thus, any coloring avoiding rainbow triangles avoids $\pi_6$, which shows $M_k(\pi_6)\geq R(K_3,\varphi_3;k)$. The upper bound follows from  the fact that any $(K_3, \pi_6, \varphi_3)$-coloring is a $(K_3,\varphi_3)$-coloring.
\end{proof}

\section{Monochromatic Cycles and Stars}

\subsection{Monochromatic Cycles}

In \cite{FM} and \cite{HMOT}, the authors produce the current best bounds for minimum order Gallai colorings forcing monochromatic cycles. In particular, they state that for all integers $k$ and $n$ with $k\geq 1$ and $n\geq 2$,

\begin{align}
(n-1)k+n+1 \leq GR_k(C_{2n}) \leq (n-1)k+3n \label{even monocycle}\\
n2^k+1\ \leq GR_k(C_{2n+1}) \leq (2^{k+3}-3)n\log n \label{odd monocycle}
\end{align}

We solve the problem for odd monochromatic cycles in the mixed case. 

By Theorem \ref{G}, any colored complete graph avoiding rainbow triangles can be partitioned into blocks of vertices where edges between two blocks are all the same color and the number of colors on edges between all blocks is two. Let the \emph{reduced graph} be the induced two-colored graph produced by taking a vertex from each block of such a \emph{Gallai partition}.  

\begin{prop}
For integers $k$ and $n$ with $k\geq 2$ and $n\geq 1$,
\begin{enumerate}
\item If $n>\frac{5^{\frac{k-1}{2}}-1}{2}$, for odd $k$, then $M_k(C_{2n+1})=R(K_3,\varphi_3;k)$\\
\item If $n>\frac{5^{\frac{k}{2}}-1}{2}$, for even $k$, then $M_k(C_{2n+1})=R(K_3,\varphi_3;k)$
\end{enumerate}
\end{prop}

\begin{proof}
We apply (\ref{mixed}) and the proof of that result given in \cite{GSSS}. Notice that every Gallai partition of $K_n$ can have no more than five blocks since the reduced graph is two colored and $R(3,3)=6$. Besides this, every color on edges in a reduced graph on a partition with more than three blocks must be distinct from the colors on edges in the blocks of that graph, to avoid a monochromatic triangle. Thus, every monochromatic cycle can be found on the edges between blocks in a fixed Gallai partition. If $2n+1>\frac{R(K_3,\varphi_3;k)-1}{2}$, for odd $k$, or $2n+1>R(K_3,\varphi_3;k)-1$, for even $k$, then we seek a monochromatic cycle of order greater than that of the blocks in any Gallai partition whose reduced graph contains a monochromatic odd cycle. Such an odd monochromatic cycle does not exist, and hence we are constrained only by the conditions of monochromatic or rainbow triangles. Hence, $M_k(C_{2n+1})=R(K_3,\varphi_3;k)$.
\end{proof}

Next we determine the function in the remaining cases for $n$ and $k$. Define 
\begin{align*}
&m_1=7\\
&m_l = 7 + \sum_{i=1}^{l-1}{(5^i\times 6)}, \mbox{ for } l>1\\
&I_1=\{3, \dots, 3+m_1\}\\
&I_l=\{3+\sum_{i=2}^{l}{(5^{i-2}\times 10)}, \dots, 3+\sum_{i=2}^{l}{(5^{i-2}\times 10)}+m_l\}, \mbox{ for } l>1\\
&i_1=3+m_1\\
&i_l=3+\sum_{i=2}^{l}{(5^{i-2}\times 10)}+m_l, \mbox{ for } l>1
\end{align*}

\begin{thm}
 For all integers $k$ and $n$ with $k\geq 2$ and $n\geq 3$, 
\[ M_k(C_{2n+1}) = 
    \left\{\begin{array}{ll}
    2^k+1, & \mbox{ when } k\geq 2, \mbox{ and } n=2\\
    5^l2^{k-2l}+1, & \mbox{ when } k\geq 2l+2, \\
    & \mbox{ and }n\in I_l \mbox{ for } l \geq 1\\
    5^l2^{k-2l}+2j+1, & \mbox{ when } k\geq 2l+2, 1\leq j \leq \frac{5^l-1}{2}, \\
    & \mbox{ and }n=i_l+j, \mbox{ for } l\geq 1
    \end{array}\right. \]
\end{thm}

\begin{proof}
The lower bound is obtained by maximizing the number of vertices in each Gallai partition of $K_N$ subject to the constraints of avoiding monochromatic triangles and odd monochromatic cycles of size $2n+1$. Let the Gallai partition of $K_N$ be the \emph{level $1$ Gallai partition}. Gallai partitions of the blocks of a \emph{level $i$ Gallai partition} are defined to be the \emph{level $i+1$ Gallai partition}, for $i \geq 1$. Note that, as in the proof of the previous proposition, the number of blocks in any Gallai partition is no greater than five. Furthermore, colors of edges in blocks must be distinct from the colors of edges between blocks, other than in the case of three blocks, in order to avoid monochromatic triangles. Thus, the number of levels of Gallai partitions can be no more than $\left\lceil\frac{k}{2}\right\rceil$ and all monochromatic odd cycles can be found on edges between blocks in fixed levels of Gallai partitions.

We begin by describing the operations of a maximal construction. Fix a level $l$ so that $0<l\leq \left\lfloor\frac{k}{2}\right\rfloor$ and suppose that the Gallai partitions at all levels at most $l$ are ``full", that is, that the order of the graph is $N=5^{l}$. The maximum size of any monochromatic odd cycle is $5^l$, which can be found by traversing the edges between blocks of the level $1$ Gallai partition. Next, notice that further expanding the graph by substituting it into vertices of a colored $K_4$ or $K_3$ without monochromatic or rainbow triangles, requires two new colors for each such substitution, and increases the order of the graph by a factor of four. Similarly, substituting the vertices into a $K_2$ requires one new color and increases the order of the graph by a factor of two. Furthermore, a substitution into $K_4, K_3,$ or $K_2$ does not produce monochromatic odd cycles of size greater than $5^l$. Perform a substitution into $K_4$, $l'$ times, and depending on the parity of $k$ either perform a substitution into $K_2$ one time or not at all, so that 
\begin{align}
l+l'=\frac{k}{2} \mbox{ when } k \mbox{ is even and } l+l'+1=\frac{k+1}{2} \mbox{ when } k \mbox{ is odd}. \label{parts}
\end{align}

Further define the levels of the graph as $-1$ for the first substitution into $K_4$, incrementing by $-1$ for every successive substitution, for levels $-1,\dots, -l'$. If $k$ is odd, add an additional level $-l'-1$ by substituting into $K_2$.

Next, introduce $2j+1$ new vertices for $0\leq j \leq \frac{5^l-1}{2}$ adjacent to all other vertices, and color the edges incident to the new vertices with existing colors without creating $K_3$ or $\varphi_3$. Any new vertex $v$ must belong to a block at some level, say $m$. Since levels $1,\dots, l$ all have five blocks, $v$ cannot belong to a new block on any of these levels. However, if $v$ is in an existing block at each such level, then $v$ is in a block at level $l$, where all blocks are composed of single vertices. This means that, with the inclusion of $v$, the number of blocks on this level is $6$, which is impossible. Since the number of blocks on the negative levels is either four or two, $v$ is in a block of one of the negative levels, which did not exist before the inclusion of $v$. 


We perform the above partitioning while keeping track of the maximum length of any monochromatic odd cycle. Let $K_N$ be partitioned into four blocks at every level. Notice that $N=2^k$ and there are no monochromatic odd cycles of any length. Let $K_N$ be partitioned into five blocks at level $1$, four blocks at level $2$, and further as described in (\ref{parts}). Notice that $N=5\times 2^{k-2}+1$ and there are no odd monochromatic cycles of size greater than $5$. Introducing one new vertex produces monochromatic cycles of sizes $7,9,\dots, 21$. Introducing two additional vertices produces a monochromatic cycles of size $23$ and two more vertices produce a cycle of size $25$.

We repeat this procedure, successively increasing the number of levels where we partition with five blocks, and then filling in the rest with four block substitutions followed by a two block substitution as necessitated in (\ref{parts}). Next we introduce $2j+1$ vertices for $0\leq j \leq \frac{5^l-1}{2}$ as described above. 

If we partition into five blocks up to level $l$, for some positive integer $l$, and substitute into four blocks and two blocks as in (\ref{parts}), and introduce one additional vertex, then we produce monochromatic cycles of sizes $2n+1$ for $n \in I_l$.

If we introduce a total of $2j+1$ additional vertices for $1 \leq j \leq \frac{5^l-1}{2}$, then we produce monochromatic cycles of sizes $2n+1$ for $n=i_l+j$.

The vertex maximality of this construction is implied by Theorem \ref{G} and the fact that the maximum number of blocks in any Gallai partition is five. 
\end{proof}

The mixed Ramsey numbers for even monochromatic cycles seems more difficult, though it is easy to see that the correct order is linear in $k$. Indeed, lexical colorings of complete graphs contain no monochromatic cycles and $M_k(C_{2n}) \leq GR_k(C_{2n})$. Finding the exact value may depend on the answer to the following

\begin{q}
For $1\leq i \leq 5$, let $B_i$ be a simple graph of $b_i$ independent vertices. What is the maximum order of an even cycle on $\left(\vee_{i=1}^{5}{B_i}\right)\vee B_1$?
\end{q}

\subsection{Monochromatic Stars}

The question of the minumum order or a complete Gallai-colored graph forcing a fixed monochromatic star was solved in \cite{GSSS} as Theorem 5. Moreover, due to the simplicity of that argument, the difficulties encountered in the mixed case are surprising.

\begin{q}
For integers $p,k>2$ and $H$ a monochromatic star $K_{1,p}$, determine $M_k(H)$.
\end{q}

Although we cannot answer this question for all $k$, the situation is clear when the number of colors is large enough.

Define
\begin{align*}
g(p)= \left\{     \begin{array}{ll}       \frac{5p}{2}-3, & \text{if }p\text{ is even}\\[1ex]
       \frac{5p-3}{2}, & \text{if }p\text{ is odd}     \end{array}   \right.\\
\end{align*}

\begin{rem}
It has been shown (Theorem $5$ of \cite{GSSS}) that any Gallai-colored complete graph of order at least $g(p)$ must contain a monochromatic $K_{1,p}$.
\end{rem}

\begin{prop}\label{starlargek}
For any integer $p>2$ there exists $K>0$ so that for all $k\geq K$,
$M_k(K_{1,p})=\min\{R(K_3,\varphi_3;k),g(p)\}$
\end{prop}

\begin{proof}
The proof is a mixture of the arguments from Theorem $5$ of \cite{GSSS}, forcing monochromatic stars, and the proof of (\ref{mixed}) from the same paper.

The upper bound follows from the definition of $M_k(K_{1,p})$. Thus, it is enough to construct a coloring on complete graphs of order one less than the upper bound, which avoid monochromatic and rainbow triangles, as well as monochromatic $K_{1,p}$.

We consider the reduced graph of a $(K_3,K_{1,p},\varphi_3;k)$-colored $K_n$ as $K_5$ with two monochromatic cycles of different colors. For odd $p$, let each block of this graph be of order $\frac{p-1}{2}$ and for even $p$, let one block be of order $\frac{p}{2}$ and the other four blocks of order $\frac{p}{2}-1$. If $k$ is large enough to allow for a mixed coloring in each block, and the sum of block sizes does not exceed the pure mixed coloring bounds, we obtain the required coloring. 

\end{proof}

For smaller values of $k$, we believe the bound is related to the function $\Ya(K_3,\varphi_3;k)$.

\section{Colored Four-Cliques}

Mixed Ramsey numbers of colored $K_4$ can be obtained without further work from mixed Ramsey numbers on colored four-cycles in all but a few cases. Indeed, it is easy to see that if we color $K_4$ by first coloring a four- cycle by one of the patterns $\pi_1, \pi_2, \pi_4, \pi_5,$ or $\pi_6$, either the colors of the remaining edges are forced, or it is impossible to avoid monochromatic or rainbow triangles. In either case, the resulting mixed Ramsey number is identical to that for the inital colored cycle. 

The above observation leaves us to consider two cases of $K_4$ with a cycle colored as $\pi_3$. We show these remaining colorings below.

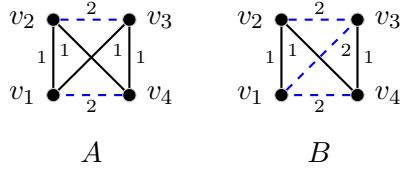
\begin{center}
\begin{tikzpicture}[>=stealth]
\tikzstyle{vertex}=[circle,minimum size=5pt,inner sep=0pt]
\tikzstyle{edge} = [draw,thick,-,black]
\tikzstyle{selected edge} = [draw,thick, -,red!50]
\tikzstyle{dash} = [draw,thick,-,style=dashed,blue]

\node[vertex] (u1) at (0,0) [label= left:$v_1$][circle,draw=black!20,fill=black!]{};
\node[vertex] (u2) at (0,1)[label=left:$v_2$][circle,draw=black!20,fill=black!]{};
\node[vertex] (u3) at (1,1)[label=right:$v_3$][circle,draw=black!20,fill=black!]{};
\node[vertex] (u4) at (1,0)[label=right:$v_4$][circle,draw=black!20,fill=black!]{};	
\draw[edge] (u1) -- (u2)  (u3) -- (u4) (u1) -- (u3) (u2) -- (u4);
\draw[dash] (u2) -- (u3)  (u4) -- (u1);

\node at (-.15,.5) {\tiny{$1$}};
\node at (.5,1.15) {\tiny{$2$}};
\node at (1.15,.5) {\tiny{$1$}};
\node at (.5,-.15) {\tiny{$2$}};
\node at (.15,.6){\tiny{$1$}};
\node at (.85,.6){\tiny{$1$}};

\node at (.5,-.75) {$A$};

\node[vertex] (u1) at (3,0) [label= left:$v_1$][circle,draw=black!20,fill=black!]{};
\node[vertex] (u2) at (3,1)[label=left:$v_2$][circle,draw=black!20,fill=black!]{};
\node[vertex] (u3) at (4,1)[label=right:$v_3$][circle,draw=black!20,fill=black!]{};
\node[vertex] (u4) at (4,0)[label=right:$v_4$][circle,draw=black!20,fill=black!]{};	
\draw[edge] (u1) -- (u2)  (u3) -- (u4) (u2) -- (u4);
\draw[dash] (u2) -- (u3)  (u4) -- (u1) (u1) -- (u3);

\node at (2.85,.5) {\tiny{$1$}};
\node at (3.5,1.15) {\tiny{$2$}};
\node at (4.15,.5) {\tiny{$1$}};
\node at (3.5,-.15) {\tiny{$2$}};
\node at (3.15,.6){\tiny{$1$}};
\node at (3.85,.6){\tiny{$2$}};

\node at (3.5,-.75) {$B$};




\end{tikzpicture}
\end{center}

\begin{thm}\label{A} 
For all integers $k\geq 2$, 
\[ M_k(A) = 
    \left\{\begin{array}{ll}
\frac{5}{3}k+2, & \mbox{ if } k \equiv 0 \pmod 3 \\[1ex]
\frac{5}{3}(k-1)+3, & \mbox{ if } k \equiv 1 \pmod 3\\[1ex]
\frac{5}{3}(k+1)+1, & \mbox{ if } k \equiv 2 \pmod 3\\
    \end{array}\right. \]
\end{thm}

\begin{proof}
For the lower bound we maximize the number of vertices in each Gallai partition while avoiding the pattern $A$. We describe the first two levels of the Gallai partitions and generalize the construction. For $k=2$, color $K_5$ with two monochromatic five-cycles for a level $1$ Gallai partition and call the colored graph $B_1$. For $k=3$, let $B_2$ be the graph consisting of one vertex. Consider two blocks, $B_1$ and $B_2$, with edges between them colored by the third color not previously used. For $k=4$, let $B_2'$ be the graph consisting of two vertices adjacent by an edge colored by a third color not used in $B_1$ and the edges between $B_1$ and $B_2'$ be colored by a fourth color not previously used. For $k=5$, let $B_2''$ be a $K_5$ with two monochromatic five-cycles colored by a third and fourth color not previously used. Color the edges between $B_1$ and $B_2''$ by a fifth new color. Continue the construction in this way, for $B_3, B_3',$ and $B_3''$, so that the reduced graph on three blocks is a two-colored triangle. Further, the reduced graph on four blocks is $B$ and the reduced graph on five blocks is $K_5$ with two monochromatic five-cycles.

This defines two levels of a Gallai partition. For every next level, we repeat the previous constructions.

To show the upper bound, observe that if two blocks in a Gallai partition have edges of the same color, then $A$ must exist between them. Thus, every block must have edges of distinct colors in order to avoid $A$. Therefore, the above construction provides the maximum number of vertices for any given $k$.  

\end{proof}

\begin{prop}\label{B}
For all integers $k\geq 2$,
$M_k(B)= M_k(\pi_1)=2^k+1$
\end{prop}

\begin{proof}
It is easy to see that given a $\pi_1$-colored cycle, to avoid $K_3$ and $\varphi_3$, the remaining edges must be colored by the minority color of the cycle. The resulting graph is isomorphic to $B$. Thus, any mixed coloring avoiding $B$ must also avoid $\pi_1$, and the reverse statement holds as well.
\end{proof}

\section{Questions About General Bounds}

Based on some general arguments using Gallai's theorem and empirical evidence from our results, we state a few open questions.

\begin{enumerate}
\item Determine colorings $f$ of $C_n$ so that $M_k(f(C_n))$ is a decreasing piecewise constant function, linear, or exponential in $k$.\\
\item Show that for any coloring $f$ of $C_n$, $M_k(f(C_n))$ can only be a decreasing piecewise constant function, linear, or exponential in $k$.\\
\item If a coloring of a graph $H$ does not contain a "large" monochromatic star, show that $M_k(H)$ is either linear or exponential in $k$ for large enough $k$.
\end{enumerate}

\end{document}